\documentclass[final]{IEEEtran}
\IEEEoverridecommandlockouts   
\IEEEoverridecommandlockouts                              % This command is only
\usepackage{amsmath,amssymb,mathtools}
\usepackage{amsthm,thmtools}
\usepackage{amsfonts} 

\usepackage{floatrow}%
\usepackage{enumerate}
\usepackage{color}
\usepackage{verbatim}
\usepackage{amssymb}
\usepackage{amsbsy}
\usepackage{amsmath}
\usepackage{setspace}
\usepackage{url}
\usepackage{todonotes}
\usepackage{bbm}

\newcommand{\diag}{\operatorname{diag}}

\newcommand{\st}{\, | \,}

\newcommand{\trace}{\operatorname{trace}}

\declaretheorem[qed=$\blacksquare$ ]{Example}

\declaretheorem[name={Theorem}  ] {Theorem}
\declaretheorem[name={Lemma}  ] {Lemma}
\declaretheorem[name={Remark}  ] {Remark}

\declaretheorem[name={Proposition}  ] {Proposition}

\newcommand {\R}{\mathbb R}

\usepackage{hyperref}       % hyperlinks
\hypersetup{colorlinks=true, linkcolor=blue, breaklinks=true, urlcolor=blue}

\DeclareMathAlphabet{\mymathbb}{U}{BOONDOX-ds}{m}{n}

\newcommand{\be}{\begin{equation}}
\newcommand{\ee}{\end{equation}}

\newcommand{\osimplex}{\mathring{\Delta}}

\title{
Minimum effort decentralized control design \\ for contracting network systems}
\author{Ron Ofir, Francesco Bullo, and Michael Margaliot\thanks{RO is with the Andrew and Erna Viterbi Faculty of Electrical and Computers Eng., Technion---Israel Institute of Technology, Haifa 3200003, Israel.
FB is with the Dept. of Mechanical Eng. and the Center for Control, Dynamical Systems, and Computation, University of California, Santa Barbara, 93106-5070, USA.
MM (michaelm@tauex.tau.ac.il) is with the School of Elc. Eng. - Systems, Tel Aviv University, 69978, Israel.
The research of MM is partially supported by research grants from the~ISF and the~DFG. The work of FB was supported in part by AFOSR grant FA9550-22-1-0059.} }
\date{Mar. 2022}

\begin{document}
\maketitle
\begin{abstract}
%%%%
  We consider the problem of making a networked system contracting by
  designing ``minimal effort'' local controllers.  Our method combines a
  hierarchical contraction characterization and a matrix-balancing approach
  to stabilizing a Metzler matrix via minimal diagonal perturbations.  We
  demonstrate our approach by designing local controllers that render
  contractive a network of
  FitzHugh–Nagumo 
 neurons with a general topology
  of interactions.
%%%%%
\end{abstract}


%%%%%%%%%%%%
\section{Introduction}
Many models of systems   consist of smaller  sub-systems which interact with each other over a network. In particular, such models are often of a large scale, in the sense that they are described by a large number of state variables. There is a renewed interest in such systems due to applications in  social dynamics, the power grid, neuroscience, and more. 

In large-scale networked systems, it is desirable to design controllers that stabilize the system using local measurements only, and using as little control effort as possible. For example, in the context of pandemic control, this corresponds to minimizing the use of protective resources~\cite{Preciado2014optimal_resource}, or minimizing the negative effect of lockdowns on the economy~\cite{Ma2022optimal}. Naturally, the control design algorithm must also be computationally efficient. 

One powerful approach for the analysis and control synthesis of large-scale
nonlinear systems is contraction
theory~\cite{LOHMILLER1998683,sontag_cotraction_tutorial}. Contractivity
implies a well-ordered asymptotic behaviour: if the system is
time-invariant and admits an equilibrium then the equilibrium is globally
exponentially stable (see, e.g.~\cite{sontag_cotraction_tutorial}). If the
system is time-varying and~$T$-periodic then contraction implies
entrainment, that is, every state variable in the network converges to a unique $T$-periodic solution (see,
e.g.~\cite{LOHMILLER1998683,entrain2011,RFM_entrain}). This property is
important in many natural and artificial systems ranging from power
electronics to systems biology. Furthermore, there exist easy to check
sufficient conditions for contraction of networked systems based on matrix
measures~\cite{Russo2013hier_contraction}.

The main contribution of this paper is a new approach for 
the computationally efficient design of ``minimum-effort'' local controllers 
guaranteeing that the closed-loop network system is 
contracting with a specified contraction rate. 
% The approach provides an optimal solution in some cases.  
We demonstrate our approach by designing the control in a network of FitzHugh–Nagumo~(FHN) neurons, with a general interaction topology, so that the network is contractive. Our approach provides conditions guaranteeing that the local controllers make the closed-loop network contractive and thus guaranteeing entrainment to periodic excitations. 
This property plays an important role in a multitude of sensory and cognitive processes~\cite{LAKATOS2019R890}.

Our approach brings together in a creative new way two recent results. The first is a sufficient condition for contraction of a nonlinear networked system, that appeared  in~\cite{Davydov2021noneuclidean} in the context of contraction with respect to norms induced by weak pairings. This sufficient condition turns the question of contraction to that of checking whether a certain Metzler matrix is Hurwitz. The second result is  a method for finding the minimal diagonal perturbation required to stabilize a Metzler matrix, presented in~\cite{Ma2022optimal} in the context of optimal lockdown design for controlling pandemics. This method is based on an elegant reduction of the optimization problem to a matrix balancing problem that can be solved using efficient algorithms. As this result is very recent, we provide here a self-contained review. 

We use the following notation. $\R_{\ge0}^n$ [$\R_{>0}^n$] is the subset of vectors in~$\R^n$ with non-negative [positive] entries. For~$A \in \R^{n \times n}$, $\alpha(A)$ denotes the spectral abscissa of~$A$, i.e. the maximal real part of the eigenvalues of~$A$. 
A matrix~$M\in\R^{n\times n}$ is called [marginally] Hurwitz if~$\alpha(M)<0$ [$\alpha(M)=0$]. 
%%%%%%%
A matrix~$M\in\R^{n\times n}$ is called Metzler if all its off-diagonal entries are non-negative. This is equivalent to the fact that the flow of~$\dot x=Mx$ maps~$\R_{\geq 0 }^n$ to itself i.e., the linear system is positive.  
Dynamical systems that admit an invariant cone are called positive or monotone, and it is well-known that for such systems stability analysis and control synthesis tend to scale well with the system dimension~\cite{RANTZER201572}. 
For~$A,B\in\R^{n\times m}$ we write~$A \leq B$ if~$a_{ij} \leq b_{ij}$ for all~$i,j$. Let~$\mathbbm{1}_n \in\R^n$ denote the vector with~$n$ entries equal to~$1$.

%%%%%%%%%%%%%%
\section{Problem formulation}
%%%%%%%%%%%%%%
Consider a networked system consisting of $m$ time-varying subsystems 
\begin{equation}\label{eq:network}
    \dot x^i(t) = f^i(t,x^1(t),\dots,x^m(t)) - u_i(t) x^i(t),\quad i=1,\dots,m, 
\end{equation}
where $x^i \in \Omega^i \subseteq \R^{n_i}$ and $u_i \in \R$. Note that~$u_ix^i$ may be interpreted as a local controller in subsystem~$i$, with  a stabilizing effect when~$u_i$ is positive.

We assume that~$f^i$ is continuously differentiable, and that~$\Omega^i$ is convex for any~$i\in\{1,\dots,n\}$. Let $n := \sum_{i=1}^m n_i$, $\Omega := \Omega^1 \times \dots \times \Omega^m$, and 
\[
    x := \begin{bmatrix}
        x^1 \\
        \vdots \\
        x^m
    \end{bmatrix}. 
\]
Then $x \in \Omega \subseteq \R^n$. Let $u := \begin{bmatrix}u_1 & \dots & u_m \end{bmatrix}^T  $, and denote 
$
    \delta_{ij}:=\begin{cases} 1, &i=j,\\
                                0, & i\not =j.
\end{cases}
$
  
The derivative of the vector field
in~\eqref{eq:network} with respect to~$x^j$
is
\be\label{eq:jij}
    J^{ij}(t,x,u) := \frac{\partial f^i}{\partial x^j}(t,x) - \delta_{ij} u_i(t) I_{n_i}.
\ee
Let
\begin{equation}\label{eq:net_jacobian}
    J(t,x,u) := \begin{bmatrix}
        J^{11}(t,x,u) & \cdots & J^{1m}(t,x,u) \\
        \vdots & \ddots & \\
        J^{m1}(t,x,u) & & J^{mm}(t,x,u)
    \end{bmatrix}.
\end{equation}
Fix~$\eta>0$. If follows from~\eqref{eq:jij} that if 
all the~$\frac{\partial f^i}{\partial x^j} $s are uniformly bounded 
then the overall system can be made contracting with rate~$\eta$ by 
setting~$u_i(t)\equiv  c$, $i=1,\dots,m$, with~$c>0$ sufficiently large. 
This naturally yields the question of how to find a ``minimum effort control'' 
that guarantees that the networked system is contracting with rate~$\eta$.

We formalize this question by posing 
the following optimization problem. 
Given~$\eta>0$, a weight vector~$w \in \R_{>0}^m$,  and a matrix measure~$\mu:\R^{n\times n} \to \R$, consider the problem 
\begin{equation}\label{eq:opt_contract}
\begin{aligned}
%%%%    \min_{\substack{u \in \R_{\ge0}^m\\|\cdot| : \R^s \to \R_+}} \quad & w^T u, \\
    \min_{v \in \R_{>0}^m } \quad & w^T v, \\
    \mathrm{s.t.} \quad & \mu(J(t,x,v)) \le -\eta \text{ for all } t \ge 0,\; x \in \Omega . 
\end{aligned}
\end{equation}
 In other words, the goal is  to find constant controls~$u_i(t)\equiv v_i$, $i=1,\dots,m$,  
guaranteeing  that the network system is contractive with rate~$\eta$, while minimizing the ``total cost'' $w^T v$. 
In particular, by setting~$w_i\gg w_j$ for all~$j \not =i$,   we can try to find a solution that guarantees a small control 
effort~$u_i(t)\equiv v_i$ in the~$i$th controller, 
if such a solution exists.

The optimization problem~\eqref{eq:opt_contract} 
is difficult to address directly because the  constraint on~$\mu(J)$
has to hold everywhere in the state-space and for all time. Furthermore, the matrix measure~$\mu$ is itself a decision variable of the problem, and it is not clear how to choose a ``good''~$\mu$. 

The  approach we propose here 
overcomes these difficulties by: (1)~replacing the constraint by a stronger condition which only requires that a certain \emph{constant} 
Metzler matrix is (marginally) stable. This removes the need to study the Jacobian directly, and essentially makes the choice of matrix measure implicit; and (2) efficiently solving  
the resulting optimization problem   using matrix balancing. 

The remainder of this note is organized as follows. The next section reviews several known definitions and results that are used later on. Section~\ref{sec:main}
describes our main results. Section~\ref{sec:appli} demonstrates  an application of our  theoretical   results to   a network of FHN neurons, and the final section concludes.

\section{Preliminaries}
%%%%%%%%%%%%%%%%%%%%%%%%%%%%%%%%%%%%%%%%%%%%%%%%%%%%%%%%%%%
We first review   known results that will be used later on. 
%%%%%%%%%%%%%%%%
\subsection{Sufficient condition for contraction in networked systems}
%%%%%%%%%%%%%%%%%%%%%%%
We briefly review a result by Str{\"o}m~\cite{Strom1975} which gives an upper bound for the matrix measure of a block matrix~$A$ based on the matrix measure of a smaller matrix $B$, where each entry of $B$ corresponds to a single block of~$A$. Given~$x\in\R^n$, decompose it as 
\begin{equation}
    x = \begin{bmatrix}
        x^1 \\
        \vdots \\
        x^m
    \end{bmatrix},\;\;x^i \in \R^{n_i},\;\;\sum_{i=1}^m n_i = n.
\end{equation}
Let $|\cdot|_i$ denote a norm on $\R^{n_i}$, and let~$|\cdot|_0$ denote a \emph{monotonic} norm\footnote{A vector norm~$|\cdot|:\R^n\to\R_{\ge0}$ is called  monotonic if~$|y_i| \leq |x_i|$ for all~$i = 1,\dots,n$ implies that~$|y| \leq |x|$; see~\cite{Bauer1961_mono_norms} for more details.} on~$\R^m$. Define a norm~$|\cdot|:\R^n \to \R_+$ by
\begin{equation}\label{eq:hier_norm}
    |x| := \left|
    \begin{bmatrix}
    |x^1|_1 \\ \vdots \\ |x^m|_m
    \end{bmatrix}
    \right|_0.
\end{equation}
Given $A \in \R^{n \times n}$, partition it into blocks $A^{ij} \in \R^{n_i \times n_j}$, with~$i,j\in\{1,\dots,m\} $, and define their induced matrix norms by
$
    \|A^{ij}\|_{ij} := \sup_{ z \in \R^{n_j} \setminus\{0\} }  |A^{ij} z|_i / |z|_j.
$
%%%%%%%%%%%%%%%%%%%%%%%
\begin{Theorem}{\cite{Strom1975}} \label{thm:strom}
%%%%%
Let~$\mu$ denote the matrix measure induced by the norm~$|\cdot|$ defined in~\eqref{eq:hier_norm}. 
Let~$\mu_i$ denote the matrix measure induced by $|\cdot|_i$, $i = 0,\dots,m$.
Define $B \in \R^{m \times m}$ by
\[
    B_{ij} := \begin{cases}
        \mu_i(A^{ii}), & i=j, \\
        \|A^{ij}\|_{ij}, & i \neq j.
    \end{cases}
\]
   Then
$        \mu(A) \le \mu_0(B)$.
\end{Theorem}
Thus, if~$\mu_0(B)\leq -\eta<0$ then~$\mu(A)\leq -\eta<0$. Note   that~$B$ is    Metzler   by construction.

\subsection{Matrix balancing}
%%%%%%%%%%%%%%%%%
A non-negative matrix $A \in \R^{n \times n}_{\geq 0}$ is  called  \emph{balanced} (some authors use the term  \emph{sum-symmetric}~\cite{Bapat1997nonnegative}) if 
$
    A \mathbbm{1}_n = A^T \mathbbm{1}_n.
$
In other words, the sum of the entries in row~$i$ of~$A$ is
equal to the sum of entries in column~$i$ of~$A$, for all~$i=1,\dots,n$. 
For example, every symmetric matrix is balanced. Also, every doubly stochastic matrix is balanced, as the sum of every row and every column is one.
 
A Metzler matrix~$A\in\R^{n\times n} $ is said to be \emph{balancable via diagonal similarity scaling} (BDSS) if there exists a diagonal matrix~$D \in \R^{n \times n}$, with positive diagonal entries, such that~$D^{-1} A D$ is balanced. The following result from~\cite{Kalantari1997balancing}
presents a sufficient condition for~BDSS, and shows that balancing is equivalent to solving an optimization problem. 
%This formulation as an optimization problem will be important when we relate balancing to the question of stabilizing Metzler matrices using minimal effort. 
Balancing is typically presented for non-negative matrices.
We state this result in the slightly more general setting of Metzler matrices. The proof   is   in the appendix.
%%%%%%%%%%%%%%%%%%%%%%%%%%%%%%%%
\begin{Theorem}[Balancing Theorem]\label{thm:diag_balance} \cite{Kalantari1997balancing}
%%%%%%%%%%%%%%%%%%%%%%%%%%%%%%%%%%%%%%%%%%%%%%%%%%%%
    Let~$A \in \R^{n \times n}$ be Metzler and irreducible. Define~$f : \R_{>0}^n \to \R$ by
    \begin{equation}\label{eq:def_fd}
        f(d) := \mathbbm{1}_n^T (\diag(d))^{-1} A \diag(d) \mathbbm{1}_n , 
    \end{equation}
    and consider the optimization problem
     \begin{equation}\label{eq:minfd}
        \min_{d \in \R_{>0}^n} f(d).
    \end{equation}
    Then:
    \begin{enumerate}
        \item There exists a $d^* \in \R_{>0}^n$ 
         that
         is a solution of~\eqref{eq:minfd};
        \item $A$ is~BDSS and in particular~$(\diag(d^*))^{-1} A \diag(d^*)$ is balanced; and
        \item If $\bar d, d^* \in \R_{>0}^n$ are solutions of~\eqref{eq:minfd}, then $\bar d = c d^*$ for some~$c > 0$.
    \end{enumerate}
\end{Theorem}
\begin{Remark}
%%%%%%%%%%%%%%%%%%%%%%%%%%%%%%%%%%%%%%%%
 A matrix is called  \emph{completely reducible} if it 
 is permutation-similar to a block-diagonal matrix, where each block is \emph{irreducible}. Equivalently, the graph corresponding to a completely reducible matrix is a union of strongly connected graphs.
    Several  recent papers  state that  irreducibility is a necessary and sufficient condition for~BDSS. This is wrong. For example,   the identity matrix   is~BDSS, but   not irreducible.  The correct statement is:
    a non-negative matrix is~BDSS if and only if it is  completely reducible. 
    Many of the results in this note which assume  irreducibilty (including Prop.~\ref{prop:strom_contraction}, Thm.~\ref{thm:diag_balance} and Lemma~\ref{lem:olshevsky_margin_stab} and Thm.~\ref{thm:olshevsky_stab_bal} below)
    are easily extended to the more general case of complete reducibility.
\end{Remark}
%%%%%%%%%%%%%%%
%%%%%%%%%%%%%%%
\begin{Remark} \label{rem:diagonal-invariance}
    The diagonal entries of~$A$
    do not affect the balancing:   if $D^{-1}AD$ is balanced,  with~$D$   a positive diagonal matrix, then for any diagonal matrix~$P$, $D^{-1}(A+P)D$ is also balanced.
\end{Remark}
%%%%%%%%%%%%
There exist efficient  numerical algorithms for matrix balancing that, under certain conditions, run in nearly linear time in the number of non-zero entries of the matrix, see~\cite{cohen2017matrix}. 
Matrix balancing   is a  useful   preconditioning step in many  matrix algorithms, and
procedures for  matrix balancing are often included in numeric computing software (e.g.,  the procedure  \texttt{balance} in MATLAB).

In some cases, there are closed-form expressions for the positive diagonal matrix~$D$ which balances~$A$. The following well-known result (see, e.g.~\cite[Ch.~0]{total_book}) gives such an expression for tridiagonal matrices. Note that a tridiagonal matrix is irreducible if and only if all entries on the super- and sub-diagonal are non-zero.
%%%%%%
\begin{Proposition}\label{prop:tridiag_balance}
    Let~$A \in \R^{n \times n}$ be Metzler, irreducible and tridiagonal. Define the positive diagonal matrix~$D \in \R^{n \times n}$ by~$d_{11} := 1$, and~$d_{ii} := \sqrt{\prod_{j=1}^{i-1}\frac{a_{j+1,j}}{a_{j,j+1}}}$ for~$i \geq 2$. Then~$D^{-1}AD$ is symmetric and thus  balanced.
%%%%%%%%%%
\end{Proposition}

%%%%%%%%%%%%%%%%%%%%%%%%%%%%%%%%%%
\subsection{Marginal stability of a Metlzer matrix}
%%%%%%%%%%%%%%%%%%%%%%%%%%%%%%%%%%%%%%%%%%%%%%%%%%%
There are several well-known characterizations  of when a Metlzer matrix is Hurwitz~\cite{berman87,bullo_graph_metzler}. For our  purposes, we need    the following condition for marginal stability of   a Metzler matrix. For the sake of completeness, we include the proof. 
%%%%%%%%%%%%%%%%%%%%
\begin{Lemma}\label{lem:olshevsky_margin_stab}
    Let $A \in \R^{n \times n}$ be Metzler and irreducible. Then~$\alpha(A) \le 0$ iff there exists~$d \in \R_{>0}^n$ such that~$Ad \le 0$.
\end{Lemma}
\begin{proof}
    Suppose that there exists $d \in \R_{>0}^n$ such that $Ad \le 0$. Then,
    $
        A \diag(d) \mathbbm{1} = Ad \le 0,
   $
      so the sum of every row  of the matrix~$B:=(\diag(d))^{-1}A\diag(d)$ is non-positive. Since $A$ is Metzler, so is~$B$ ,and thus 
$        b_{ii} + \sum_{\substack{j=1\\j\neq i}}^n |b_{ij}| = \sum_{j=1}^n b_{ij} \le 0,\quad i=1,\dots, n. 
   $
    By Gershgorin's Theorem~\cite[Thm.~6.1.1]{Horn2013matanalysis}, all eigenvalues of~$B$ lie in the closed left half plane. This implies that the same holds for the eigenvalues of~$A$. 
    
    To prove the converse implication, assume that~$\alpha(A) \le 0$. Since $A$ is Metzler and irreducible, there exists $r \ge 0$ such that~$
        S := A + r I 
 $
    is   irreducible and non-negative.
    By the Perron-Frobenius Theorem~\cite[Thm.~8.4.4]{Horn2013matanalysis}, $S$ has a real eigenvalue $\lambda >0$ and corresponding  eigenvector~$d\in \R_{>0}^n$. By the assumption,~$\lambda \le r$. Thus,~$        S d = \lambda d \le r d.
$    This gives~$A d \le 0$, and this completes the proof.
\end{proof}

\subsection{Minimal effort diagonal  stabilization of Metzler matrices}
%%%%%%%%%%%%%%%%%%%%%%%%%%%%%%%%%%%%%%%%%%%%%%%%%%%%%%%%%%%%%%%%%%%%%%%
We now review  a minimal effort  controller design for an irreducible positive LTI system based on matrix balancing. To the best of our knowledge, this idea first appeared in~\cite{Ma2022optimal} in the context of optimal lockdown design for pandemic control. With respect to~\cite[Theorem~4.6 and its proof]{Ma2022optimal}, the following theorem statement is more general (e.g., it allows for general Metzler matrices, arbitrary target spectral abscissa, 
and for  the diagonal perturbation to take negative values), more explicit (e.g., an explicit formula for the diagonal perturbation is given as a function of the balancing diagonal matrix), and establishes additional properties of the transcription (e.g., the Perron eigenvector of the closed-loop system); additionally, the proof is more concise.
  
\begin{Theorem}\label{thm:olshevsky_stab_bal}
    Let $A \in \R^{n \times n}$ be Metzler and irreducible.
    Fix weights~$w \in \R_{>0}^n$ and target spectral abscissa~$\eta \in \R$.
    Let~$d^* \in \R_{>0}^n$ be such that the matrix 
    \[
    (\diag(d^*))^{-1} \diag(w) A \diag(d^*)
    \]
    is balanced and define 
    %%%%%%%%%
    \begin{equation}\label{eq:opt_cont}
        \ell^* := (\diag(d^*))^{-1} A \diag(d^*) \mathbbm{1}_n - \eta \mathbbm{1}_n.
    \end{equation}
    Then
    \begin{enumerate}
        \item the Metzler matrix $A - \diag(\ell^*)$ has spectral abscissa~$\eta$ and Perron eigenvector~$d^*$.
        \item $\ell^*$ is the unique solution of the optimization problem
        \begin{equation}\label{eq:optim_stab}
        \begin{aligned}
            \min_{\ell \in \R^n} \quad & w^T \ell, \\
            \mathrm{s.t.} \quad & \alpha(A - \diag(\ell)) \le \eta.
        \end{aligned}
        \end{equation}
        \item If~$A - \eta I \ge 0$ then $\ell^* \in \R_{>0}^n$.
    \end{enumerate}
\end{Theorem}

For~$\eta\leq 0$ the goal of problem~\eqref{eq:optim_stab} is to guarantee that
the spectral abscissa of~$B:=A-\diag(\ell)$  is smaller or equal to~$\eta $, so in particular the irreducible matrix~$B$ is (marginally)  Hurwitz. This should be done with the ``smallest possible'' diagonal perturbation~$\ell$ in the sense that~$w^T\ell$ is minimized. 
There is considerable literature 
on finding the closest
Metzler and Hurwitz matrix to a given matrix (see~\cite{closest_metlzer} and the references therein), but the advantages of the formulation in~\eqref{eq:optim_stab}
are:
(1)~$\alpha(A-\diag(\ell)) $ is convex in~$\ell$~\cite{cohen1981convexity};
and
(2)~as we will see below, it can be naturally interpreted as finding ``minimal effort'' local controllers that render a network contractive.

Note that~$w$ does not appear explicitly in the formula for~$\ell^*$, but the vector~$d^*$ there does depend on~$w$. 

%%%%%%%%%%%%%%%%%%%%
\begin{proof}
%%%%%%%%%%%%%%%%%%%
	Since $A$ is Metzler and irreducible and $w\in\R^n_{>0}$, $\diag(w)A$ is also Metzler and irreducible, and by  Thm.~\ref{thm:diag_balance} it is~BDSS.
	
    We now show that $\ell^*$ in~\eqref{eq:opt_cont}
    is the optimal solution to~\eqref{eq:optim_stab}.
    %%%%%%
    By Thm.~\ref{thm:diag_balance} and Remark~\ref{rem:diagonal-invariance},~$d^*$ in the theorem statement is a minimizer of
    \begin{equation}\label{eq:optim_stab_ww}
        \min_{d \in \R_{>0}^n} f(d),
    \end{equation}
    with
    \begin{align*}\label{eq:optim_stab_ww}
        f(d) &:=
        \mathbbm{1}_n^T  (\diag(d))^{-1}\diag(w) (A - \eta I) \diag(d) \mathbbm{1}_n \nonumber \\
        &= w^T  (\diag(d))^{-1} (A - \eta I) \diag(d) \mathbbm{1}_n.
    \end{align*}
    Since~$w \in \R_{>0}^n$,~$f(d) \le w^T \ell$
    for any~$\ell \in \R^n$ such that~$\ell \ge (\diag(d))^{-1} (A - \eta I) \diag(d)\mathbbm{1}_n$. Furthermore, as~$d \in \R_{>0}^n$,
    \begin{align*}
        & (\diag(d))^{-1} (A - \eta I) \diag(d) \mathbbm{1}_n \le \ell \\
        & \iff (A - \eta I)  d  \le \diag(d)\ell = \diag(\ell) d \\
        & \iff (A - \diag(\ell))d \le \eta d.
    \end{align*}
    We conclude that~\eqref{eq:optim_stab_ww} can be rewritten as 
    \begin{equation}\label{eq:optim_stab_aux2}
    \begin{aligned}
        \min_{\substack{\ell \in \R^n\\d \in \R_{>0}^n} } \quad & w^T \ell, \\
        \mathrm{s.t.} \quad & (A - \diag(\ell))d \le \eta d,
    \end{aligned}
    \end{equation}
    and optimal solutions to~\eqref{eq:optim_stab_aux2} must satisfy the equality~$(A - \diag(\ell))d = \eta d$. Then, by Thm.~\ref{thm:diag_balance},~$(\ell^*,d^*)$ is an optimal solution to~\eqref{eq:optim_stab_aux2}. Since~$A - \diag(\ell^*)$ is Metzler and irreducible,~$(A - \diag(\ell^*))d^* = \eta d^*$ implies that~$\eta$ is the spectral abscissa of~$A - \diag(\ell^*)$ and~$d^*$ is a Perron eigenvector. This proves statement~1).
    
    By Lemma~\ref{lem:olshevsky_margin_stab},~\eqref{eq:optim_stab_aux2} is equivalent to~\eqref{eq:optim_stab}. Thus,~$\ell^*$ is an optimal solution to~\eqref{eq:optim_stab}, and it is unique by the third statement in Thm.~\ref{thm:diag_balance}. This proves statement~2).
     
    Finally, statement~3) follows from the definition of~$\ell^*$ and the fact that $A - \eta I$ is non-negative and irreducible.
\end{proof}

\begin{Example}
Consider the controlled two-dimensional flow system
$
\dot x=A x -\diag(\ell_1,\ell_2) x,
$
where~$A:=\begin{bmatrix}
-1&1\\1&-1
\end{bmatrix} f$.
Here~$f>0$ models  the flow rate between two nodes.
Since $A-\eta I_2\geq0$ holds for any~$\eta\leq -f$,  we set~$\eta:=-(f+\varepsilon)$, with~$\varepsilon\geq 0$.  Consider the optimization problem~\eqref{eq:optim_stab}  with~$w:=\begin{bmatrix}
1&w_2
\end{bmatrix}^T$, where~$w_2>0$. Then~$\diag(w)A=\begin{bmatrix}
-1&1\\w_2&-w_2
\end{bmatrix}f$, and
$
(\diag(d))^{-1} \diag(w)A \diag(d)
$
is balanced for~$d=\begin{bmatrix}1&\sqrt{w_2}
\end{bmatrix}^T$, so~\eqref{eq:opt_cont} gives
\be\label{eq:lopt}
\ell^*  =  \begin{bmatrix}
 \sqrt{w_2} f+\varepsilon &  \frac{1}{\sqrt{w_2}} f+\varepsilon 
\end{bmatrix} ^T.
\ee
The closed-loop system is then
$\dot x=A_c x$, with
\begin{align*}
 A_c&:=A-\diag(\ell^*)\\
    &=\begin{bmatrix}
                          -\varepsilon-(1+\sqrt{w_2} ) f & f\\
                          f& -\varepsilon-( 1+\frac{1}{\sqrt{w_2}}  )f
%%%%
    \end{bmatrix}.
\end{align*}
The eigenvalues of~$A_c$ are~$-(f+\varepsilon)$ and~$-(f+\varepsilon) -(\sqrt{w_2}+\frac{1}{\sqrt{w_2}})f$, so~$\alpha(A_c)=\eta$.
Note that~\eqref{eq:lopt} implies the following. \begin{enumerate}
    \item 
 If~$w_2\ll 1$
(i.e., the cost function is~$w^T \ell \approx \ell_1$) 
then~$\ell^*\approx 
\begin{bmatrix}
\varepsilon  & \frac{1}{\sqrt{w_2}} f
\end{bmatrix}^T$.
\item If~$w_2=1$ (i.e., the cost function is~$w^T \ell = \ell_1+\ell_2$) 
then~$\ell^*= 
\begin{bmatrix}
  f  +\varepsilon& f+\varepsilon
\end{bmatrix}^T.$
\item
If~$w_2\gg 1 $
(i.e., the cost function is~$w^T \ell \approx w_2 \ell_2$) 
then~$\ell^*\approx 
\begin{bmatrix}
\sqrt{w_2} f  & \varepsilon
\end{bmatrix}^T.$
%%%%%%%%%%%%%%%%%%%
 \end{enumerate}

%%%%%%%%%%%
\end{Example}

%%%%%%%%%%%%%%%%%%%%%%%%%%%

\section{Main results}\label{sec:main}
%%%%%%%%%%
We now combine the ideas above to provide a novel,  simple and efficient algorithm for 
finding local  controllers  guaranteeing  contraction in a networked system. We begin with several auxiliary results. 

Ref.~\cite{Russo2013hier_contraction} used Thm.~\ref{thm:strom} to derive a hierarchical approach to contraction. This approach requires finding a monotonic norm under which a certain nonlinear system is contracting. This is hard to do in general, as the monotonic norm has to induce a matrix measure that is negative at every point in the state space. This approach was further simplified in~\cite{Davydov2021noneuclidean}, which derived a stronger sufficient condition (i.e., one that is  applicable for a smaller family of systems) which instead involves checking whether a certain \emph{constant}
  Metzler matrix is Hurwitz. In~\cite{Davydov2021noneuclidean} this result was stated in terms of one-sided Lipschitz constants. Here we state and prove this result using matrix measures instead. The first step is to remove the dependency of~$B$ on~$t$ and~$x$ and replace it with a constant matrix. To do so, we will make use of the fact that~$B$ is Metzler by construction, and that $|\cdot|_0$ is a monotonic norm. 
%%%%%%%%%%%%%%%%%%
\begin{Proposition}\label{prop:measure_monotone}
%%%%%
Let~$|\cdot|_0:\R^n\to \R_+$ be a monotonic vector norm, and let~$||\cdot||_0:\R^{n\times n }\to \R_+$ and~$\mu_0:\R^{n\times n}\to \R$ denote the induced matrix norm and matrix measure. 
If~$A,B \in \R^{n \times n}$ are Metzler  
   and~$A \le B$  then 
$
        \mu_0(A) \le \mu_0(B)
$.
\end{Proposition}
%%%%%%%%%%%%%%%%%%%%%%%%%%
\begin{proof}
    Since $A \le B$,  we have
$
        I + hA \le I + hB  
$ for any~$h \geq 0$.
    Furthermore, since~$A$ and~$B$ are Metzler,   we have that~$I + hA$ and~$I + hB$ are non-negative matrices
    for any~$h>0$ sufficiently 
    small. 
    By~\cite[Thm.~4]{Bauer1961_mono_norms},
    $|| I + h A ||_0\leq || I + hB ||_0$, and using the definition of the matrix measure completes the proof. 
\end{proof}

Consider now the networked system~\eqref{eq:network}. Construct $B$ from the blocks $J^{ij}$ of its Jacobian. Define a \emph{constant} matrix~$\hat{J} \in \R^{m\times m}$  by
\begin{equation}\label{eq:J_sup}
    \hat{J}_{ij} := \begin{cases}
        \sup_{\substack{x \in \Omega\\t \ge 0}} \mu_i(J^{ii}(t,x)), & i=j, \\
        \sup_{\substack{x \in \Omega\\t \ge 0}} \|J^{ij}(t,x)\|_{ij}, & i \neq j.
    \end{cases}
\end{equation}
By construction, $\hat{J}$ is Metzler and $B(t,x) \le \hat{J}$   for all $t \ge 0$ and $x \in \Omega$. This leads to the following result.
\begin{Proposition}\label{prop:strom_contraction}
%%%%%%%%%%%%%%%%%%%%%%
Let $\mu$ denote  the matrix measure induced by the norm defined in~\eqref{eq:hier_norm}.
Fix~$\varepsilon>0$.  There exists a monotonic norm~$|\cdot|_0$ with induced matrix measure $\mu_0(\cdot)$ such that
    \begin{equation}\label{eq:mono_norm_abscissa}
      \mu(J(t,x)) \leq \mu_0(\hat{J}) \le \alpha(\hat{J}) + \varepsilon,  \text{ for all } t\geq0,x\in\Omega. 
    \end{equation}
 %%%%%
    In particular, if $\hat{J}$ is Hurwitz then the network~\eqref{eq:network} is contracting with rate~$\alpha(\hat{J}) + \varepsilon$.
    %%%%
    If in addition $\hat{J}$ is irreducible, then~\eqref{eq:mono_norm_abscissa} 
    holds with $\varepsilon=0$.
\end{Proposition}
%%%%%%%%%%%%%
\begin{proof}
    Eq.~\eqref{eq:mono_norm_abscissa} follows from~\cite[Thm.~2]{Strom1975}. The fact that the system is contracting if $\hat{J}$ is Hurwitz then follows by Thm.~\ref{thm:strom}.
    %%%%%%%%%%
  If~$\hat{J}$ is irreducible (and  Metzler)  the  Perron-Frobenius Theorem~\cite[Thm.~8.4.4]{Horn2013matanalysis} implies that~$\alpha(\hat{J})$ is a simple eigenvalue of $\hat{J}$. Now~\cite[Thm.~3]{Strom1975} implies that~\eqref{eq:mono_norm_abscissa} holds also for~$\varepsilon=0$.
  %%%%%%%%%%%%%%%%
\end{proof}

We now apply Prop.~\ref{prop:strom_contraction} to obtain a sufficient condition for contraction in the networked system~\eqref{eq:network}. Since~$\mu(A + \alpha I) = \mu(a) + \alpha$ for any matrix measure and any~$\alpha \in \R$, Eq.~\eqref{eq:jij} gives
%%%%%%%%%%%
$
    \mu(J^{ii}) = \mu(\frac{\partial f^i}{\partial x^i})  - u_i.
$
Therefore, a sufficient condition for~\eqref{eq:network} to be contracting is that the Metzler matrix $\hat{J}  -  \diag(u)$ is Hurwitz. Combining this with Thm.~\ref{thm:olshevsky_stab_bal} yields the following result for determining an upper bound on the effort required to guarantee that~\eqref{eq:network} is contracting.
%%%%%%%%%%%%%%%%%%%%%
\begin{Theorem}\label{thm:net_contract}
%%%%%%%%%%%%%%%%%
    Consider the networked system~\eqref{eq:network} and define~$\hat J\in\R^{m\times m}$ as in~\eqref{eq:J_sup}. Suppose that~$\hat J$ is irreducible. Fix~$w \in \R_{>0}^m$ and $\eta > 0$ such that $\hat J + \eta I_m \ge 0$.
    Then  there exists a~$d \in \R_{>0}^m$ such that $(\diag(d))^{-1} \diag(w) \hat{J} \diag(d)$ is balanced, and
    \begin{equation}\label{eq:opt_controller_contract}
         v  ^* := (\diag(d))^{-1} \hat{J} \diag(d) \mathbbm{1}_m + \eta \mathbbm{1}_m
    \end{equation}
    is the optimal solution to
    \begin{equation}\label{eq:suboptimal_contract}
    \begin{aligned}
        \min_{v \in \R_{>0}^m} \quad & w^T v, \\
        \mathrm{s.t.} \quad & \alpha(\hat{J} - \diag( v )) \le -\eta.
    \end{aligned}
    \end{equation}
    Furthermore, $v ^*$ is a feasible solution of~\eqref{eq:opt_contract}.
\end{Theorem}
%%%%%%%%%%%%%%%%%%
Thm.~\ref{thm:net_contract} shows that finding      local controllers   guaranteeing  that~\eqref{eq:network} is contracting with rate~$\eta$ can be done by diagonally balancing an upper-bound of the reduced order Jacobian of the system. Since diagonal  balancing can be solved  efficiently, this approach  is useful even in the case of large-scale systems.

In certain cases, Thm.~\ref{thm:net_contract} may in fact yield the optimal solution to~\eqref{eq:opt_contract}. Consider the networked system~\eqref{eq:network} and suppose that all subsystems are scalar  i.e.,~$n_i = 1$ for all~$i$. If there exists~$x \in \Omega$ such that~$J(x) = \hat J$ and $\hat J$ is irreducible, then~$v^*$ in~\eqref{eq:opt_controller_contract} gives the minimal controller which guarantees contraction with rate $\eta$ under \emph{any} constant norm. Indeed, a necessary condition for contraction with respect to a constant norm in such systems  is that~$\hat J - \diag(u)$ is Hurwitz. Thm.~\ref{thm:olshevsky_stab_bal} guarantees that the optimal controller stabilizing~$\hat J$ is~$ v^*$ in~\eqref{eq:opt_controller_contract}, and Prop.~\ref{prop:strom_contraction} guarantees that this controller also achieves contraction. Note that a special case of such a system is an irreducible positive~LTI system.

\begin{Example}
Consider the network system~\eqref{eq:network}, and suppose that $\frac{\partial f^i}{\partial x^j} \equiv 0$ for all $i,j$ such that $|i-j|>1$, so~$\hat J$ is a tridiagonal Metzler matrix. Assume in addition that $\hat J$ is irreducible, and let~$\eta > 0$ be such that $\hat J + \eta I_m \ge 0$. By Prop.~\ref{prop:tridiag_balance} and Thm.~\ref{thm:net_contract}, the optimal  controller~$v^*  $ for~\eqref{eq:suboptimal_contract} is %%%
\[
    v^*_i = \eta +\hat J_{ii}+ \begin{cases}
      \sqrt{\hat J_{1,2} \hat J_{2,1}}, & i = 1, \\
      \sqrt{\hat J_{i,i+1} \hat J_{i+1,i}} + \sqrt{\hat J_{i-1,i} \hat J_{i,i-1}}, & 1<i<m, \\
          \sqrt{\hat J_{m-1,m} \hat J_{m,m-1}}, & i = m,
    \end{cases}
\]
and~$v^* \in \R_{>0}^m$.
%%%
This can be explained as follows. 
The optimal control~$v^*_i$  amounts to ``canceling'' the diagonal term~$\hat J_{ii}$ and also ``canceling'' the effect of its four ``neigbours'', and then subtracting~$\eta$.
\end{Example}
 
%%%%%%%%%%%%%%%%
\section{An application}\label{sec:appli}
%%%%%%%%%%%%%%
We apply our approach   to design local controllers in a network of  
 FHN neurons that was studied using hierarchical  contraction in~\cite{Russo2013hier_contraction}.
The FHN model is a simplified 2D version of the detailed Hodgkin–Huxley model 
for the 
activation and deactivation dynamics of a spiking neuron.
We derive sufficient conditions under which the network is contractive, and thus entrains to periodic inputs. This application in fact shows that our approach does not necessarily require 
considering a Metzler matrix based  on  hierarchical  contraction, but can also be applied in other cases.

The network  consists of~$N$ neurons, each modeled according to the FHN model
\begin{equation}\label{eq:fi_neuron}
\begin{aligned}
    \dot v_i &= c \left(v_i + w_i - \frac{1}{3}v_i^3 + r(t) \right) + h_i(v), \\
    \dot w_i &= - (v_i - a + b w_i)/c,
\end{aligned}
\end{equation}
for $i = 1,\dots,N$, where $v_i$ denotes the membrane voltage, $w_i$ is a recovery variable, $r(t)$ is an external input current, and~$v:=\begin{bmatrix}
v_1&\dots&v_N\end{bmatrix}^T$. Here~$a,b\geq 0$ and~$c>0$. The function~$h_i(v)$ describes a connection  term:
\begin{equation}\label{eq:hivterm}
    h_i(v) = \gamma \sum_{j\in\mathcal{N}_i}(v_j - v_i) - \ell_i v_i,
\end{equation}
where~$\gamma>0$, $\mathcal{N}_i$ is the set of neighbours of neuron~$i$, and~$\ell_i>0$ is the gain of an additional local control term, which we will determine next such that the network is contractive. 

Let~$x:=\begin{bmatrix}
    v_1&\dots&v_N&w_1&\dots& w_N
\end{bmatrix}^T$. Then the 
Jacobian  of the dynamics 
is 
  %%%%%%%
\begin{equation}
    J(x) = \begin{bmatrix}
        J^{11}(v) - \diag(\ell) & cI_N \\
        - I_N /c & -bI_N/c
    \end{bmatrix},
\end{equation}
where~$J^{11}(v) := c I_N - c (\diag(v))^2 - \gamma L$, and~$L\in\R^{N \times N}$ is the Laplacian of the graph describing the interactions between the neurons, that is,
\[
    L_{ij} = \begin{cases}
        |\mathcal{N}_i| , & i = j, \\
        -1               , & i \neq j \text{ and } j\in \mathcal{N}_i, \\
        0               , & \text{otherwise}.
    \end{cases}
\]
The matrix~$J(x)$ is not Metzler, but rather than applying Thm.~\ref{thm:net_contract} at this point, we will use the fact that $J(x)$ can be transformed to a  skew-symmetric form to guarantee  contraction under a scaled~$L_2$ norm. Let
$
    T := \begin{bmatrix}
        I_N & 0 \\
        0 & cI_N
    \end{bmatrix},
$
and define a scaled~$L_2$ norm by:~$|x|_{2,T} := |Tx|$. Then,
\begin{align*}
    \mu_{2,T}(J ) &= \mu_2(T J T^{-1}) \\
    &= \mu_2\left(\begin{bmatrix}
        J^{11}(v) - \diag(\ell) & I_N \\
        -I_N & - {b} I_N/c
    \end{bmatrix}\right) \\
    &= \mu_2\left(\begin{bmatrix}
        S(v) - \diag(\ell) & 0 \\
        0 & - {b} I_N/c
    \end{bmatrix}\right) \\
    &= \max\{\mu_2(S(v)- \diag(\ell)), - {b}/ c \},
\end{align*}
where~$S(v):=(J^{11}(v)+(J^{11}(v))^T)/2$ is the symmetric part of~$J^{11}(v)$. 
Since~$S(v)$ is Metzler for any~$v$,   Prop.~\ref{prop:measure_monotone}  gives  
\[
    \mu_2(S(v)) \le \mu_2(\hat J^{11}),\text{ for all } v,
\]
where~$\hat J^{11} := S(0)= cI_N - \gamma (L + L^T)/2$.
Therefore, a sufficient condition for contraction with rate~$\eta \in [0, {b}/{c}]$ w.r.t. the scaled~$L_2$ norm $|\cdot|_{2,T}$ is that $\mu_2(\hat J^{11}) = -\eta$. Furthermore, since $\hat J^{11}$ is symmetric, $\mu_2(\hat J^{11}) = \alpha(\hat J^{11})$. For any~$\eta \ge \gamma\max_i \{L_{ii}\} - c$, the matrix~$\hat J^{11} + \eta I_N$ is non-negative, so by Thms.~\ref{lem:olshevsky_margin_stab} and~\ref{thm:net_contract} the minimal~$\ell$ guaranteeing that~$\hat J^{11}$
is Hurwitz with~$\alpha(\hat J^{11}) = -\eta$ (and thus the network is contractive with rate~$\eta$) is %%%%%%%%%%
\begin{align}\label{eq:fhn_opt}
    \ell^* &= (c+\eta) \mathbbm{1}_N - \frac{\gamma}{2} (L+L^T) \mathbbm{1}_N  \nonumber \\
  & = (c+ \eta) \mathbbm{1}_N - \frac{\gamma}{2} L^T \mathbbm{1}_N  .
\end{align}
%%%%%%%
Note that for any~$i$ the required
control effort~$\ell_i^*$ decreases with~$\gamma \sum_{j\neq i} (L_{ij} - L_{ji})$ (i.e., when the connections are stronger or when neuron~$i$ is fed by more neurons or feeds less neurons).
The control effort increases with~$c$, as a larger~$c$ means that~\eqref{eq:fi_neuron} is less stable, and with the required rate of contraction~$\eta$. Also, if the interconnection is symmetric  then 
\be\label{eq:lstar_homog}
\ell^* = (c + \eta)\mathbbm{1}_N,
\ee
so the optimal controller is independent of the network topology. This is not surprising, as~$-L$ is marginally stable if the network is symmetric, so all that is needed to stabilize the system is to ``cancel'' the unstable effect   of~$c$. When~\eqref{eq:lstar_homog} 
holds, \eqref{eq:fi_neuron}, 
and~\eqref{eq:hivterm} 
imply that the diagonal set~$\{x: v_i=v_j , w_i=w_j \text{ for all } i,j\}$ is an  invariant set of the closed-loop network. 
Since the network is also contractive, 
this implies that the neurons do not only entrain, but also synchronise, that is, 
$
 v(t) \to  
  \beta_1(t) 1_N 
$
and
$w(t)\to
  \beta_2(t) 1_N 
$,
where every~$\beta_i(t)$ is a scalar~$T$-periodic function. 

Fig.~\ref{fig:entrainment} depicts  the membrane voltage of three of the neurons (to avoid cluttering) in the non-symmetric  network of six neurons used in~\cite[Fig.~1]{Russo2013hier_contraction} with 
the~$T$-periodic input~$r(t) = 4 + 4\sin(2\pi t)$, for~$T=1$, and the controller~$\ell^*$   in~\eqref{eq:fhn_opt}. It may be seen that all the neurons entrain to the periodic input. 

\begin{figure}
    \centering
    \includegraphics[scale=0.8]{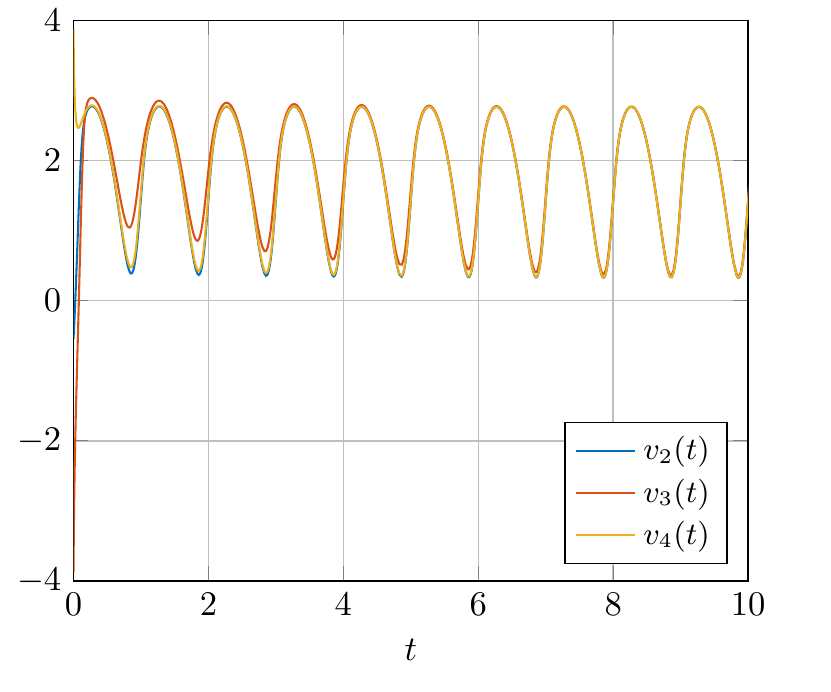}
    \caption{Entrainment  in a network of   FHN neurons with the 
     interconnection topology in~\cite{Russo2013hier_contraction}. The parameters are $a=0,b=2,c=6,\gamma=0.05,\eta=0.05$, and~$\ell^*$ as in~\eqref{eq:fhn_opt}.}
    \label{fig:entrainment}
\end{figure}

\section{Conclusion}
%%%%%%%%%%%%%%%%%%%%%%%%%%%5
We considered  the problem of efficiently designing local controllers which guarantee that a large-scale network   becomes contractive, while keeping the total control effort minimal.
We addressed this problem by first 
attaining a constant Metzler  matrix~$B$ 
such that making~$B$ Hurwitz implies contractivity of the network, and then 
using an efficient algorithm, based on matrix balancing, for determining the minimal diagonal perturbation making~$B$ Hurwitz~\cite{Ma2022optimal}.

Matrix balancing is a well studied topic with many 
generalizations~\cite{ideal2016,Eaves1985}. It may  be interesting   to use this to derive more
general versions of the optimization problem~\eqref{eq:optim_stab}. 
Another direction for further research is 
to consider generalized versions of~\eqref{eq:opt_contract} which require contraction with respect to a space- and time-dependent norm, rather than a  constant norm.
%%%%%%%%%%%%%%%%%%%%%%%%%%%%%%%%%%%%%%%%%%
\subsection*{Appendix: Proof of Thm.~\ref{thm:diag_balance}}
%%%%%%%%%%%%%%%%%%%%%
First,~$f(d)$ is homogeneous of degree zero, i.e., $f(c d) = f(d)$ for any $c > 0$, so we may restrict our attention to the set~$\osimplex := \{d \in \R_{>0}^n \st \sum_i d_i = 1\}$. Consider the optimization problem
\begin{equation}
\begin{aligned}
    \min_{d \in \R_{>0}^n} \quad & f(d), \\
    \mathrm{s.t.} \quad & \sum_i d_i = 1.
\end{aligned}
\end{equation}
%%%%%%%%%%%%%%%%%%%
Fix a vector~$\mathring{d}\in \osimplex$  satisfying  that the set~$Z$   of indexes~$i$ such that~$\mathring{d}_i=0$ is not empty.    Then~$\bar Z :=\{1,\dots,n\} \setminus Z$   is also  non empty. Since~$A$ is irreducible, there exist~$i \in Z$ and~$j \in \bar Z$ such that~$a_{ij} > 0$. Then
$
    f(d) \ge \trace(A) + a_{ij} {d_j} d_i^{-1},
$
so~$\lim_{d \to \mathring{d}} f(d) = \infty$. Since~$f$ is continuous in~$\osimplex$, it attains a minimal value there. This proves the assertion in~1).
  
To prove the second assertion, note that since $d\in\R^n_{>0}$, we can define a vector~$g\in \R^n$ by~$g_i := \ln(d_i), i=1,\dots,n$. Then~\eqref{eq:minfd} can be rewritten as 
    \begin{equation}\label{eq:optim_balance_conv}
        \min_{g \in \R^n} \tilde{f}(g),
    \end{equation}
where
    \begin{align*}
        \tilde{f}(g): &= \mathbbm{1}_n^T \exp(-\diag(g)) A \exp(\diag(g)) \mathbbm{1}_n\\
       & = \sum_{i,j} a_{ij} \exp(g_j-g_i)  
    % \\  &= \sum_i a_{ii} + \sum_{i\neq j} a_{ij}\exp(g_j-g_i).
    \end{align*}
Since~$a_{ij} \ge 0$ for any~$i \neq j$,  $\tilde f$ is   a sum of convex functions, so it is convex. Therefore,~\eqref{eq:optim_balance_conv} is convex and  unconstrained, so the minimum is achieved at any point~$g^*$ where the gradient~$\frac{\partial }{\partial g} \tilde f (g^*)$ vanishes. This is equivalent to~$\exp(-\diag(g^*))  A\exp( \diag(g^*))
=(\diag(d^*))^{-1} A \diag(d^*)
$ being a  balanced matrix.
    
To prove the third assertion, let~$p,q\in\R_{>0}^n $, with~$p \not = q$,  be two minimizers of~\eqref{eq:optim_balance_conv}. Define~$v(\varepsilon):=\varepsilon p+(1-\varepsilon)q$. Then
    \begin{align*}
        \frac{  d^2 }{d \varepsilon^2}\tilde f(v(\varepsilon)) &= \sum_{i\neq j} 
         a_{ij} ( p_j-q_j+q_i-p_i )^2  \exp(
         v_j(\varepsilon)-v_i(\varepsilon)) \\
         &=\sum_{ i\neq j} 
         a_{ij} ( r_j-r_i )^2  \exp(
         v_j(\varepsilon)-v_i(\varepsilon))\\
         &\geq 0 ,  
    \end{align*}
where~$r:=p-q$.
If~$ \sum_{i \neq j} 
a_{ij} ( r_j-r_i )^2>0$ then~$\frac{  d^2 }{d \varepsilon^2}\tilde f(v(\varepsilon))>0$ for any~$\varepsilon>0$, so~$\tilde f(v(1/2))< \tilde f(v(0))$ which is a contradiction. We conclude that
    \be\label{eq:allzerp}
    \sum_{i,j} 
         a_{ij} ( r_j-r_i )^2=0. 
    \ee
Hence,  there exists a set of indexes~$I\subseteq   \{1,\dots,n\}  $, with~$ |I|\geq 2$,
such that~$r_{i_1}=r_{i_2}$ for any~$i_1,i_2 \in I$, and~$r_{i}\not = r_{j} $ for any~$i\in I,j\in\bar I:=\{1,\dots,n\}\setminus I  $. 
Suppose that~$\bar I$ is not empty. Since~$A$ is irreducible and Metzler, there exist~$i\in I$ and~$j\in \bar I$ such that~$a_{ij}>0$, and this contradicts~\eqref{eq:allzerp}.
Thus,~$I=\{1,\dots,n\}$. Hence, if~$p\not = q$ are two minimizers of~$\tilde{ f}$ then~$p=q+c \mathbbm{1}_n$ for some~$c\not =0$. 
This completes the proof of Thm.~\ref{thm:diag_balance}.
%%%%%%%%%%%%%%%%%%%%%%%%
\subsection*{Acknowledgements} We thank Rami Katz and Chengshuai Wu for helpful comments.

%\bibliographystyle{IEEEtran}
%\bibliography{literature}
%%%%%%%%%%%%%%%%%%%%%%%%%%%%%%%%%%%%%%%
% Generated by IEEEtran.bst, version: 1.14 (2015/08/26)

\end{document}